\documentclass[12pt,a4paper,mathscr]{amsart}
\usepackage[a4paper]{geometry}
\usepackage{mathptmx} 
\DeclareMathAlphabet{\mathcal}{OMS}{cmsy}{m}{n} 

\usepackage{amssymb,eucal,mathrsfs,mathbbol}

\newtheorem{theorem}{Theorem}[section]
\newtheorem{lemma}[theorem]{Lemma}
\newtheorem{corollary}[theorem]{Corollary}
\newtheorem*{question}{Question}

\theoremstyle{remark}
\newtheorem*{remark}{Remark}
\newtheorem*{remarks}{Remarks}

\newcommand{\set}[2]{\ensuremath{\{ #1 \>|\> #2 \}}}

\def\form{\ensuremath{(\,\cdot\, , \cdot\,)}}

\DeclareMathOperator{\Ker}{Ker}
\DeclareMathOperator{\SO}{SO}

\begin{document}

\title{On the utility of Robinson--Amitsur ultrafilters}
\author{Pasha Zusmanovich}
\address{}
\email{pasha.zusmanovich@gmail.com}
\date{last (cosmetical) revision May 31, 2016}
\thanks{J. Algebra \textbf{388} (2013), 268--286; arXiv:0911.5414} 

\begin{abstract}
An embedding theorem for algebraic systems is presented,
basing on a certain old ultrafilter construction. 
As an application, we outline alternative proofs of some results from the 
theory of PI algebras, and establish some properties of Tarski's monsters.
\end{abstract}

\maketitle

\section*{Introduction}

In 1960s, A. Robinson and S. Amitsur established a number of embedding results in Ring Theory
which proved to be useful in various structural questions. A typical example:
if a prime ring $R$ embeds in a direct product of associative division rings, 
then $R$ embeds in an associative division ring 
(see \cite[Proof of Theorem 15]{amitsur-1} and \cite[Theorem 3]{amitsur-2} for
original papers, and \cite[\S 6]{eklof} for a nice overview).

The proof of these results follows the same scheme:
basing on the initial data -- a ring embedded in a direct product of rings --
a certain sort of ultrafilter, which we call a \textit{Robinson--Amitsur ultrafilter}, is 
constructed. Using this ultrafilter, one passes from the direct product of rings to their
ultraproduct, and appeal to the {\L}o\'s theorem about elementary equivalence of 
an algebraic system and its ultraproduct completes the proof.

In this paper we extend this argument to 
a class of general algebraic systems (Theorem \ref{ultra}), 
and observe similarity with the 
celebrated J\'onsson lemma from the universal algebra.
Coupled with the classical Birkhoff theorem about varieties of algebraic systems, 
this gives a simple yet elegant criterion for an algebra or group not to satisfy a 
nontrivial identity
(Corollaries \ref{no-ident} and \ref{cor2}).
The corresponding group result is not entirely new (see comments after 
Corollary \ref{cor2}), but, we believe,
its proof is, and the links between apriori unrelated concepts, ideas and results
is the main novelty of this paper.

As an application, we outline alternative, ``by abstract nonsense'', proofs of some 
particular cases of well-known results from the theory of PI algebras (\S \ref{pi}), 
of results about algebras having the same identities (\S \ref{same-ident}), and 
establish that Tarski's monsters without identities have infinite (relative) 
girth (\S \ref{tarski}).
Finally, in \S \ref{dual} we formulate and prove a ``dual'' version of
Theorem \ref{ultra}, with embeddings replaced by surjective homomorphisms,
and discuss some its consequences and related questions.

The narrative is occasionally interspersed with questions and speculations.

\section{Algebraic systems}\label{sect-univalg}

In what follows, by an \textit{algebra} or \textit{ring}, we mean an arbitrary, 
not necessarily associative, or Lie, or satisfying any other distinguished identities, 
algebra or ring, unless it is stated otherwise. Algebras are considered over fields.
\textit{Ideal} of an algebra means a two-sided ideal.
By an \textit{(algebraic) system} we mean an algebra in the universal 
algebraic sense, i.e. a set with a number of operations on it of, generally, various 
arity. 

We deal with algebraic systems whose congruences behave ``good enough'',
like ideals in rings, or normal subgroups in groups (or, more generally,
ideals in the so-called $\Omega$-groups introduced by P.J.~Higgins and studied by 
Kurosh and his school (see \cite[Chapter 3, \S 2]{kurosh}). Namely,
suppose the signature $\Omega$ of a class of algebraic systems has a 
$0$-ary operation $e$ (i.e., a distinguished element).
A term $t(x_1, \dots, x_n, y_1, \dots, y_m)$ composed of the operations in $\Omega$ 
is called an \textit{ideal term in $y_1, \dots, y_m$} if $m>0$ and
$$
t(a_1, \dots, a_n, e, \dots, e) = e
$$
for any $a_1, \dots, a_n\in A$ for any system $A$
in the class. A subset $I$ of an algebraic system $A$ is called an \textit{ideal} if
for any ideal term $t(x_1, \dots, x_n, y_1, \dots, y_m)$ in $y_1, \dots, y_m$, 
the element
\begin{equation}\label{term}
t(a_1, \dots, a_n, b_1, \dots, b_m)
\end{equation}
belongs to $I$ for any $a_1, \dots, a_n \in A$ and $b_1, \dots, b_m \in I$.

For every congruence $\theta$ on $A$, its equivalence class 
$\set{a\in A}{a \equiv_\theta e}$ is an ideal of $A$. 
Generally, this relation between ideal and congruences is not one-to-one,
but if it is, the corresponding class of algebraic systems is called 
\textit{ideal-determined}.
In particular, for systems from an ideal-determined class we can speak about 
quotients by ideals instead of quotients by congruences.

This notion, along with its numerous particular cases and variations, was studied by 
Agliano, Chajda, Fichtner, Gr\"atzer, Gumm, S{\l}omi\'nski, Ursini and others 
(see, for example, \cite[Chapter 10]{cel}).

Any algebraic system $A$ from an ideal-determined class has at least two ideals --
a \textit{trivial ideal} $\{e\}$ and the whole $A$. 
The intersection of any two ideals of $A$ is an ideal. 
$A$ is called \textit{finitely subdirectly irreducible}
if intersection of any two its nontrivial ideals is nontrivial.
An \textit{ideal generated by the subset $X$} of $A$ is the minimal ideal of $A$ 
containing $X$, and it coincides with the set of all elements of the form
(\ref{term}), where $t(x_1, \dots, x_n, y_1, \dots, y_m)$ is an ideal term in 
$y_1, \dots, y_m$, and $a_1, \dots, a_n\in A$, $b_1, \dots, b_m\in X$.

Recall the construction of the ultraproduct -- not in the most general form, 
but in the form suitable for our purposes.
Let $\{A_i\}_{i\in \mathfrak I}$ be a set of algebraic systems from an ideal-determined 
class, and $\mathscr F$ a filter on the indexing set $\mathfrak I$. Then
$$
\mathcal I\Big(\prod_{i\in \mathfrak I} A_i, \mathscr F\Big) = 
\set{f \in \prod_{i\in \mathfrak I} A_i}{\set{i\in \mathfrak I}{f(i) = e} \in \mathscr F}
$$
is an ideal of the direct product $\prod_{i\in \mathfrak I} A_i$, and the quotient
by this ideal is called a \textit{filtered product} of the set 
$\{A_i\}_{i\in \mathfrak I}$ with respect to the filter $\mathscr F$, 
and is denoted by $\prod_{\mathscr F} A_i$.
In the particular case where all $A_i$'s are isomorphic to the same algebraic system $A$,
their filtered product is called a \textit{filtered power} of $A$ and is denoted by 
$A^{\mathscr F}$. When $\mathscr F$ is an ultrafilter, we speak about
\textit{ultraproducts} and \textit{ultrapowers}.

An alternative view on ultrafilters and ultraproducts is an analytic one:
ultrafilter is a probability measure taking only two values -- $0$ and $1$, and
such that every subset of the indexing set is measurable. An ultrapower is the 
collection of measurable functions from the indexing set considered up to the 
measure.

We refer to \cite{cel}, \cite{cohn}, \cite{comm}, \cite{kurosh}, or \cite{malcev}
for all necessary basic notions and results related to universal algebra, and to 
\cite{bell-slomson}, \cite{cohn}, \cite{eklof}, or \cite{malcev} again 
for ultrafilters and ultraproducts.

\begin{theorem}[Robinson--Amitsur for algebraic systems]\label{ultra}
Let $\{ B_i \}_{i\in \mathfrak I}$ be a set of algebraic systems from an 
ideal-determined class. 
If a finitely subdirectly irreducible algebraic system $A$ embeds in
the direct product $\prod_{i\in \mathfrak I} B_i$, then there is an ultrafilter 
$\mathscr U$ on the set $\mathfrak I$ such that $A$ embeds in the ultraproduct 
$\prod_{\mathscr U} B_i$.
\end{theorem}

\begin{proof}
Define 
$$
\mathscr S = \set{\set{i\in \mathfrak I}{f(i)\ne e}}{f\in A, f\ne e} .
$$
Let us verify that intersection of any two elements of $\mathscr S$ 
contains an element of $\mathscr S$. 
Let $S,T\in \mathscr S$, say, $S = \set{i\in \mathfrak I}{f(i)\ne e}$
and $T = \set{i\in \mathfrak I}{g(i)\ne e}$ for some $f, g\in A$ different from $e$.
Since $A$ is finitely subdirectly irreducible, it contains an element $u\ne e$
belonging to the intersection of ideals generated by $\{f\}$ and $\{g\}$. 
Let $i\in \mathfrak I$ such that $f(i) = e$.
Since $u = t(h_1, \dots, h_n, f, \dots, f)$ for some ideal term $t$ and 
$h_1, \dots, h_n\in A$,
\begin{multline*}
u(i) = t(h_1, \dots, h_n, f, \dots, f)(i) = 
t(h_1(i), \dots, h_n(i), f(i), \dots, f(i)) \\ = t(h_1(i), \dots, h_n(i), e, \dots, e)
= e .
\end{multline*}
Coupling this with a similar assertion for $g$, we get that
$$
S \cap T \supset \set{i\in \mathfrak I}{u(i) \ne e} \in \mathscr S.
$$
 
Thus $\mathscr S$ satisfies the finite intersection property and is contained in
some ultrafilter $\mathscr U$ on $\mathfrak I$. 
Factoring the embedding of algebraic systems
$A \hookrightarrow \prod_{i\in \mathfrak I} B_i$ by the ideal 
$\mathcal I(\prod_{i\in \mathfrak I} B_i, \mathscr U)$, 
we get an embedding of algebraic systems
$$
A \Big/ \Big(A \cap \mathcal I\Big(\prod_{i\in \mathfrak I} B_i, \mathscr U\Big)\Big) 
\hookrightarrow 
\Big(\prod_{i\in \mathfrak I} B_i\Big) \Big/ \mathcal I\Big(\prod_{i\in \mathfrak I} B_i, \mathscr U\Big)
= \prod_{\mathscr U} B_i.
$$
Let $f \in A \cap \mathcal I\Big(\prod_{i\in \mathfrak I} B_i, \mathscr U\Big)$. 
Then $\set{i\in \mathfrak I}{f(i) = e} \in \mathscr U$, and, since
$\mathscr U$ is an ultrafilter, 
$\set{i\in \mathfrak I}{f(i) \ne e} \notin \mathscr U$, and hence $\set{i\in \mathfrak I}{f(i) \ne e} \notin \mathscr S$. 
From $f \in A$ and the definition of $\mathscr S$ it follows that $f = e$. 
This shows that 
$A \cap \mathcal I\Big(\prod_{i\in \mathfrak I} B_i, \mathscr U\Big) = \{e\}$.
\end{proof}

The ultraproduct construction used in this proof mimics the old one, 
used by A. Robinson and S. Amitsur in Ring Theory, mentioned in the introduction.

The finite subdirect irreducibility of an algebraic system $A$ is 
equivalent to the following condition: if $A$ embeds in the finite direct product 
of algebraic systems $\prod_{i=1}^n B_i$, then $A$ embeds in one of $B_i$'s.
An infinite analog of this condition is \textit{subdirect irreducibility}, that is,
the condition that intersection of any (possibly infinite) set of nontrivial ideals of 
$A$ is nontrivial (or, equivalently, $A$ possesses a minimal nontrivial ideal
which is called \textit{monolith}). 
Similarly, the latter condition is equivalent to the following:
if $A$ embeds in the (possibly infinite) direct product $\prod_{i\in \mathfrak I} B_i$, 
then $A$ embeds in one of $B_i$'s. Thus, Theorem \ref{ultra} can be considered as,
perhaps, somewhat surprising statement that for ideal-determined classes,
finite subdirect irreducibility implies a sort of a weaker form of subdirect 
irreducibility.

\medskip

An application of Theorem \ref{ultra} to varieties and quasivarieties of algebraic 
systems follows.

If $A$ is an algebraic system, $Var(A)$ and $Qvar(A)$ denote, respectively, a variety 
and a quasivariety generated by $A$. Any quasivariety (and, in particular, any variety) 
possesses free algebraic systems 
(this is formulated explicitly, for example, in \cite[Chapter VI, Proposition 4.5]{cohn}
and is implicit in \cite[Chapter V]{malcev}).

\begin{corollary}\label{cor}
Let $A$ be an algebraic system from an ideal-determined class.
A finitely subdirectly irreducible free system in $Var(A)$ or $Qvar(A)$ embeds in an 
ultrapower of $A$.
\end{corollary}

\begin{proof}
Let $\mathcal F$ be a free system in $Var(A)$ which is finitely subdirectly irreducible.
According to the Birkhoff theorem, $\mathcal F = B/I$ for an ideal 
$I$ of an algebra $B$, and $B$ is a subalgebra of a direct power of $A$.
Because of the universal property of $\mathcal F$, the short exact sequence 
$\{e\} \to I \to B \to \mathcal F \to \{e\}$ splits, 
i.e. $\mathcal F$ embeds in $B$, and hence in a direct power of $A$.
Then apply Theorem \ref{ultra}.

Similarly, according to the Birkhoff-like characterization of quasivarieties due to 
Malcev (\cite[Chapter V, \S 11, Theorem 4]{malcev}), 
if $\mathcal F$ is a free system in $Qvar(A)$, then it is a subalgebra of a 
filtered power $A^{\mathscr F} = A^{\mathfrak I}/\mathcal I(A^{\mathfrak I}, \mathscr F)$ of 
$A$. Taking preimage of $\mathcal F$ with respect to the homomorphism 
$A^{\mathfrak I} \to A^{\mathfrak I}/\mathcal I(A^{\mathfrak I}, \mathscr F)$,
we get that $\mathcal F$ is a quotient of a subalgebra in 
$A^{\mathfrak I}$, and the rest of reasoning is the same as above.
\end{proof}

Compare Theorem \ref{ultra} and Corollary \ref{cor} with the celebrated J\'onsson lemma 
in the generalized form due to combined efforts of 
Freese, Hagemann, Herrmann, Hrushovski and McKenzie (see \cite[Theorem 10.1]{comm}): 
if $A$ is a subdirectly irreducible algebraic system from a modular variety 
(i.e., the congruence lattice of any system from the variety is modular)
generated by a set $\{B_i\}_{i\in \mathfrak I}$ of algebraic systems, 
then the quotient of $A$ by the centralizer of its monolith embeds in a homomorphic 
image of a subsystem of an ultraproduct $\prod_{\mathscr U} B_i$.
Note that the congruence (=ideal) lattice of any algebra from an ideal-determined class
is modular (see, for example, \cite[Remark 10.1.16]{cel}).

\begin{corollary}[Criterion for absence of non-trivial identities for algebraic systems]\label{no-ident-universal}
Let $\mathfrak V$ be a variety of algebraic systems from an ideal-determined class,
and suppose that all free systems of $\mathfrak V$ are finitely subdirectly irreducible.
Then for an algebraic system $A \in \mathfrak V$, the following is equivalent:
\begin{enumerate}
\item 
any identity of $A$ is an identity of $\mathfrak V$ (i.e., $A$ does not satisfy 
nontrivial identities within $\mathfrak V$);
\item any free system of $\mathfrak V$ embeds in an ultrapower of $A$;
\item any free system of $\mathfrak V$ embeds in a system elementarily equivalent to $A$.
\end{enumerate}
\end{corollary}

\begin{proof}
(i) $\Rightarrow$ (ii) follows from Corollary \ref{cor}.

(ii) $\Rightarrow$ (iii)
follows from the {\L}o\'s theorem about elementary equivalence of an algebraic system
and its ultrapower.

(iii) $\Rightarrow$ (i) 
An algebraic system elementarily equivalent to $A$ does not satisfy a 
nontrivial identity within $\mathfrak V$. Since the latter is the first-order property, 
$A$ does not satisfy a nontrivial identity either.
\end{proof}

\begin{remarks}\hfill

(1) Of course, the equivalence of conditions (ii) and (iii) also follows from 
the (powerful) Keisler ultrapower theorem (see, for example, \cite[Chapter 7, Corollary 2.7]{bell-slomson}).

(2)
Recall a well-known fact from model theory: an algebraic system $B$ embeds in
an ultrapower of an algebraic system $A$ if and only if 
$Th_\forall(A) \subseteq Th_\forall(B)$,
where $Th_\forall$ denotes the universal theory of a system
(see, for example, \cite[Chapter 9, Lemma 3.8]{bell-slomson}).
This allows to rephrase condition (ii) or (iii) as follows: 
the universal theory of $A$ is contained in 
the universal theory of any free system of $\mathfrak V$.

(3)
As any variety is determined by its free system of countable rank (see, for example,
\cite[Chapter IV, Proposition 3.8]{cohn} or \cite[Chapter VI, \S 13, Theorem 3]{malcev}),
in conditions (ii) and (iii) of the corollary, as well as in (2) above, 
one may replace ``any free system'' by ``the free system of countable rank''.
\end{remarks}

\section{Algebras}\label{sect-alg}

Specializing results of the previous section to the (ideal-determined, obviously)
class of rings, we get, for example, a nonassociative analog of one of the classical 
embedding results in Ring Theory mentioned in the introduction:

\begin{corollary}
If a finitely subdirectly irreducible ring $A$ embeds in a direct product of division 
rings, then $A$ embeds in a division ring.
\end{corollary}

\begin{proof}
By Theorem \ref{ultra}, $A$ embeds in an ultraproduct of division rings.
As the property to be a division ring is the first-order property, by 
the {\L}o\'s theorem the ultraproduct of division rings is a division ring, 
whence the conclusion.
\end{proof}

When trying to specialize to the class of algebras, we should deal with the 
issue of the base field. For example, Corollary \ref{no-ident-universal} is not applicable 
directly. The problem is this: Corollary \ref{no-ident-universal} is obtained
by combination of the Birkhoff theorem about varieties of algebras 
(or a similar statements), and the {\L}o\'s theorem about elementary equivalence
of an algebra and its ultraproduct. One may treat algebras either as two-sorted theories
(algebra over a field, field), or distinguish elements of the base field by an unary 
predicate. Either way, while
in Birkhoff-like reasonings the base field remains the same, in {\L}o\'s ones it, 
generally, changes: we pass to an ultraproduct of the base field.

So, for now on, fix the base field $K$. 
The embedding claimed in Theorem \ref{ultra} is an embedding of algebras
defined over $K$. 
Of course, the ultraproduct $\prod_{\mathscr U} B_i$ is defined also over the
ultrapower field $K^{\mathscr U}$,
so we have an embedding of $K^{\mathscr U} A$ in $\prod_{\mathscr U} B_i$ as 
$K^{\mathscr U}$-algebras. 
(Here $K^{\mathscr U} A$ is understood as a $K^{\mathscr U}$-linear span of 
$A$, considered as a $K$-subalgebra of $\prod_{\mathscr U} B_i$).
Due to the universal property of the tensor product, there is a surjection of 
$K^{\mathscr U}$-algebras
\begin{equation}\label{surj}
A \otimes_K K^{\mathscr U} \to K^{\mathscr U} A ,
\end{equation}
but this surjection is, generally, not a bijection.

An important observation is that for free algebras $A$ in the major 
varieties of algebras considered in the literature -- the varieties
of all algebras, associative algebras, and Lie algebras, the surjection (\ref{surj}) is 
a bijection.

\begin{lemma}\label{embed}
Let $A$ be a subalgebra of an algebra $B$, both defined over a field $K$.
Suppose $B$ is also defined over a field $F$ containing $K$, and that
as a $K$-algebra, $A$ does not have commutative subspaces (i.e., subspaces all whose
elements commute) of dimension $>1$. 
Then $FA \simeq A \otimes_K F$ as $F$-algebras.
\end{lemma}

\begin{proof}
The claimed isomorphism follows from the fact that the linear dependence of elements of $A$ over $F$
implies their linear dependence over $K$. Indeed, consider the linear dependence
\begin{equation}\label{lindep}
f_1a_1 + \dots + f_na_n = 0,
\end{equation}
where $f_i\in F$, $a_i\in A$, and let us prove by induction on $n$
that $a_1, \dots, a_n$ are linearly dependent over $K$. For $n=1$ this is trivial.
Taking the commutator with $a_n$ of both sides of (\ref{lindep}), we get 
$$
f_1[a_1,a_n] + \dots + f_{n-1}[a_{n-1},a_n] = 0 .
$$
(we use the usual notation for the commutator: $[a,b] = ab - ba$).
By inductive hypothesis, there are $k_1, \dots, k_{n-1} \in K$, not all zero, 
such that
$$
k_1[a_1,a_n] + \dots + k_{n-1}[a_{n-1},a_n] = [k_1a_1 + \dots + k_{n-1}a_{n-1},a_n] = 0 .
$$
But since $A$ does not have commutative subspaces of dimension $>1$, 
$$
k(k_1a_1 + \dots + k_{n-1}a_{n-1}) + \ell a_n = 0
$$ 
for some $k, \ell \in K$, not both zero.
\end{proof}

Obviously, the hypothesis of this lemma is satisfied when $A$ is a free algebra or 
a free Lie algebra of rank $>1$ (in view of the Shirshov--Witt theorem). A similar,
but just a little bit more involved argument (see, for example, proof of Lemma 1
in \cite{makar-limanov}) shows that the conclusion of the lemma holds also when 
$A$ is a free associative algebra of rank $>1$.

\medskip

Another issue we should deal with when trying to apply results of the previous section
to the class of algebras, is when the hypothesis of Corollary \ref{cor} holds, i.e., when
free algebras in a variety are finitely subdirectly irreducible.
Again, this hypothesis is satisfied for absolutely free algebras, 
free associative algebras, and free Lie algebras of rank $>1$
(on the other hand, it is not satisfied for free Jordan algebras and 
free alternative algebras). 
In fact, those free algebras satisfy a stronger condition -- primeness, and, in the 
context of algebras, we will mainly focus on the latter condition instead of finite 
subdirect irreducibility.

Let $\mathfrak V$ be a variety of algebras, and $\mathcal F(X)$ is the free 
algebra in this variety generated by a set $X$.
By \textit{words} we mean elements of $\mathcal F(X)$ for some $X$.
The standard grading on $\mathcal F(X)$ is defined by length of words.
By \textit{2-nontrivial words} we mean words degrees of all whose homogeneous 
components in each of the first two indeterminates are non zero. For example,
$$
xy, \quad (xy)x - (xy)(xz), \quad (zy)x + 2(tx)y + x(yz)(tx)
$$
are 2-nontrivial words in $x,y,z,t$ (in that order), while
$$
x, \quad xz, \quad (xy)x + xz, \quad (yz)(yt) + (xy)z
$$
are not.

\begin{lemma}\label{lemma}
For an algebra $A \in \mathfrak V$, the following is equivalent:
\begin{enumerate}
\item For any two nonzero ideals $I, J$ of $A$, $IJ \ne 0$.
\item For any two nonzero elements $x, y \in A$, there is a 2-nontrivial word 
$w(\xi_1, \dots, \xi_n)$, $n \ge 2$ and elements $x_1, \dots, x_{n-2} \in A$ 
such that $w(x, y, x_1, \dots, x_{n-2}) \ne 0$.
\end{enumerate}
\end{lemma}

\begin{proof}
(i) $\Rightarrow$ (ii). Suppose there are nonzero $x, y \in A$ such that
for any 2-nontrivial word $w(\xi_1, \dots, \xi_n)$ and any $x_1, \dots, x_{n-2} \in A$, 
$w(x, y, x_1, \dots, x_{n-2}) = 0$.
Let $I$ and $J$ be ideals of $A$ generated by $x$ and $y$ respectively. Clearly $IJ = 0$, a
contradiction.

(ii) $\Rightarrow$ (i). Suppose $I, J$ are two nonzero ideals of $A$.
Taking $x \in I$ and $y \in J$, we have $w(x, y, x_1, \dots, x_{n-2}) \ne 0$ for some
2-nontrivial word $w$ and elements $x_i$ of $A$. Clearly, \newline
$w(x, y, x_1, \dots, x_{n-2}) \in I \cap J$.

Further, there is a $2$-nontrivial word $u$ and elements $y_1, \dots, y_{m-2} \in A$ 
such that
$$
u(w(x, y, x_1, \dots, x_{n-2}), w(x, y, x_1, \dots, x_{n-2}), y_1, \dots, y_{m-2}) \ne 
0 .
$$
In particular, for some monomial $m$ occurring in $u$, we have 
\begin{equation}\label{m}
m(w(x, y, x_1, \dots, x_{n-2}), w(x, y, x_1, \dots, x_{n-2}), y_1, \dots, y_{m-2}) \ne 
0 .
\end{equation}
Examining how $m$ is built up from the variables, there must be some point where
a sub-monomial containing the first variable is multiplied by a sub-monomial
containing the second variable, giving a product $pq$ where $p$ involves one of the
first two variables, and $q$ the other.

Due to (\ref{m}),
\begin{multline*}
p(w(x, y, x_1, \dots, x_{n-2}), w(x, y, x_1, \dots, x_{n-2}), y_1, \dots, y_{m-2})
\\ \times
q(w(x, y, x_1, \dots, x_{n-2}), w(x, y, x_1, \dots, x_{n-2}), y_1, \dots, y_{m-2})
\ne 0 .
\end{multline*}
Each of these factors belongs to $I \cap J$ because $w(x, y, x_1, \dots, x_{n-2})$ does,
so their product belongs to $(I \cap J)(I \cap J) \subseteq IJ$, as required.
\end{proof}

An algebra $A \in \mathfrak V$ satisfying the equivalent conditions of Lemma \ref{lemma},
is called \textit{$\mathfrak V$-prime}.
When $\mathfrak V$ is a variety of all associative algebras, this notion coincides
with the classical notion of a prime associative algebra.

Clearly, if $\mathfrak W$ is another variety and $\mathfrak V \subseteq \mathfrak W$, 
an algebra $A \in \mathfrak V$ is $\mathfrak V$-prime 
if and only if it is $\mathfrak W$-prime,
so we can simply speak about \textit{prime algebras} 
(which are prime in the variety of all algebras).

As for any two ideals $I$ and $J$, $IJ \subseteq I \cap J$, prime algebras (or rings) 
are finitely subdirectly irreducible.

Now, the claim about primeness of free algebras and free associative 
algebras is obvious, as they do not have zero divisors. 
To establish primeness of a free Lie algebra of rank $>1$, consider two nonzero ideals 
$I$, $J$ in it. 
By the Shirshov--Witt theorem about freeness of subalgebras of a free Lie algebra,
neither of $I$, $J$ can be one-dimensional, and hence we may choose
two linearly independent elements $x\in I$ and $y\in J$. 
Their commutator is nonzero, as otherwise they would
form a $2$-dimensional abelian subalgebra, which again contradicts the Shirshov--Witt
theorem. Hence $[I,J] \ne 0$. 

\medskip

Another situation when relatively free algebras are prime is described by the
following lemma.

\begin{lemma}\label{zariski}
Let $A$ be a prime algebra over an infinite field. 
Then the free algebra of infinite rank in $Var(A)$ is prime.
\end{lemma}

\begin{proof}
Let $\mathcal F(X)$ be the free algebra in $Var(A)$ freely generated by an infinite 
set $X = \{ x_1, x_2, \dots \}$. Suppose that there are nonzero elements 
$u(x_1, \dots, x_n)$, $v(x_1, \dots, x_n) \in \mathcal F(X)$ such that 
\begin{equation}\label{yoyo}
w(u, v, u_1, \dots, u_m) = 0
\end{equation}
for any 2-nontrivial word $w$ and $u_1, \dots, u_m \in \mathcal F(X)$. 
As any relation between free generators of $\mathcal F(X)$ is an identity in $A$, the 
equalities (\ref{yoyo}) are identities in $A$. Taking $u_1 = x_{n+1}, u_2 = x_{n+2}, \dots$,
we get that 
$$
w(u(x_1, \dots, x_n), v(x_1, \dots, x_n), x_{n+1}, \dots, x_{n+m}) = 0
$$
for any 2-nontrivial word $w$ and any $x_1, \dots, x_n, x_{n+1}, \dots, x_{n+m} \in A$. 
As $A$ is prime, either \newline $u(x_1, \dots, x_n) = 0$ or $v(x_1, \dots, x_n) = 0$ 
for any $x_1, \dots, x_n \in A$.

Fix some basis of $A$, take 
$(x_1, \dots, x_n) \in A^n = A \times \dots \times A$ ($n$ times) at which $u$ is
nonzero, $(y_1, \dots, y_n) \in A^n$ at which $v$ is nonzero, and consider the line
in $A^n$ of all affine linear combinations 
\begin{equation}\label{affine}
\lambda(x_1, \dots, x_n) + (1 - \lambda)(y_1, \dots, y_n) ,
\end{equation}
where $\lambda$ is an element of the base field $K$. 
On this line, the coordinates of $u$ and $v$ in terms of the chosen basis
are each given by a polynomial in $\lambda$. So, picking a $K$-linear function
on $A$ such whose value at $u(x_1, \dots, x_n)$ is nonzero and another such 
whose value at $v(y_1, \dots, y_n)$ is nonzero, we see that for all but finitely many
$\lambda\in K$, both $u$ and $v$ will be nonzero on the $n$-tuple (\ref{affine}),
a contradiction.

If $A$ is finite-dimensional, we can use a similar, but more elegant 
Zariski-topology argument instead: both sets of $n$-tuples of elements of $A$ on 
which $u$, respectively $v$, does not vanish, form a nonempty Zariski-open subset 
in $A^n$, whence they have a nonzero intersection, a contradiction.
\end{proof}

All these observations lead to a variant of the general 
Corollary \ref{no-ident-universal} for the case of algebras:

\begin{corollary}[Criterion for absence of non-trivial identities for algebras]\label{no-ident}
For an algebra $A$ belonging to one of the following varieties of algebras: 
all algebras, associative algebras, or Lie algebras, the following are equivalent:
\begin{enumerate}
\item $A$ does not satisfy a nontrivial identity;
\item any free algebra embeds in an ultrapower of $A$;
\item any free algebra embeds in an algebra elementarily equivalent to $A$.
\end{enumerate}
\end{corollary}

\begin{proof}
(i) $\Rightarrow$ (ii) follows from Corollary \ref{cor}.
Note that in view of Lemma \ref{embed} (applied to the case $F = K^\mathscr U$, 
where $K$ is the base field, and $\mathscr U$ is an appropriate ultrafilter) and remarks after 
it, the embedding of Corollary \ref{cor} can be considered as an embedding of 
$K^{\mathscr U}$-algebras.

(ii) $\Rightarrow$ (iii)
as in the proof of Corollary \ref{no-ident-universal}:
by the {\L}o\'s theorem, the pairs $(B, K)$ and $(B^{\mathscr U}, K^{\mathscr U})$
are elementarily equivalent as models of the two-sorted theory 
(algebra over a field, field).

(iii) $\Rightarrow$ (i) as in the proof of Corollary \ref{no-ident-universal}.
\end{proof}

All the remarks after Corollary \ref{no-ident-universal} are applicable. In addition, as 
the free associative, respectively Lie, algebra of countable rank embeds in the 
free associative, respectively Lie, algebra of rank $2$, for these varieties 
one can replace ``any free algebra'' in conditions (ii) and (iii) by 
``the free algebra of rank $2$''.

\medskip

Let us provide some negative examples showing that conclusions 
of some statements of this section do not always hold.

The conclusion of Lemma \ref{embed} does not hold for polynomial
algebras in $>1$ variables (i.e., for free associative commutative algebras).
Indeed, let $F = K(x)$ (the field of rational functions in $1$ variable) and
$A = K[x,y]$ (the algebra of polynomials in $2$ variables), considered as a
$K$-subalgebra of $K(x,y)$ (the field of rational functions in $2$ variables). Then 
$$
FA = K(x)[y] \simeq K[y] \otimes_K K(x) ,
$$
but
$A \otimes_K F = K[x,y] \otimes_K K(x)$.

In some other situations, the surjection (\ref{surj}) also can be very far from 
being a bijection. For example, consider the free algebra $\mathcal F$ 
of countable rank in the variety $Var(A)$ generated by a finite-dimensional prime algebra
over an infinite field $K$. By Lemma \ref{zariski}, $\mathcal F$ is prime, and by 
Corollary \ref{cor}, $\mathcal F$ embeds, as a $K$-algebra, in an ultrapower
$A^{\mathscr U} \simeq A \otimes_K K^{\mathscr U}$ (see Corollary \ref{fin-dim} below).
Hence $K^{\mathscr U} \mathcal F$ is a finite-dimensional $K^{\mathscr U}$-algebra.
On the other hand, since $\mathcal F$ is residually nilpotent 
(see, for example, \cite[\S 4.2.10]{bahturin} for the case of Lie algebras; the general case
is treated identically), the finite-dimensionality of $\mathcal F$ 
would imply nilpotency of $\mathcal F$, and hence nilpotency of $A$, a contradiction. 
Consequently, $\mathcal F$ is infinite-dimensional over $K$,
and $\mathcal F \otimes_K K^{\mathscr U}$ is infinite-dimensional over $K^{\mathscr U}$.
We will use essentially the same type of arguments below in \S \ref{same-ident}, when
describing an alternative approach to a result of Kushkulei and Razmyslov about
varieties generated by simple finite-dimensional Lie algebras.

\section{Groups}

As the class of (all) groups is, obviously, ideal-determined, Theorem \ref{ultra} and
Corollary \ref{cor} apply to it. 

Let us single out an important
condition for groups of which the finite subdirect irreducibility is a consequence:

\begin{lemma}\label{p}
A group in which any two commuting elements generate a cyclic subgroup either of
prime order, or of infinite order, is finitely subdirectly irreducible.
\end{lemma}

\begin{proof}
Let $G$ be a group with the given property, and $N$, $M$ be two normal subgroups of $G$.
Take $x\in N, y\in M$, $x,y \ne 1$. If $x,y$ do not commute, then
$1 \ne xyx^{-1}y^{-1} \in N \cap M$. 
If $x,y$ commute, then they generate a cyclic subgroup of $G$,
generated by a single element $a\in G$, either of prime order $p$, 
or of infinite order.
Write $x = a^n$, $y = a^m$ for some $0 < n,m < p$ in the first case, and some 
nonzero integers $n,m$ in the second one. We have $a^{nm} \in M \cap N$, and due to 
the restriction on $n,m$, $a^{nm} \ne 1$ in both cases. 
\end{proof}

Now we can establish Corollary \ref{no-ident-universal} for the variety of all groups
and for Burnside varieties:

\begin{corollary}[Criterion for absence of non-trivial identities for groups]\label{cor2}
For a group $G$ belonging to one of the following varieties: all groups,
groups satisfying the identity $x^p = 1$ for a prime $p \ge 673$, the following
are equivalent:
\begin{enumerate}
\item $G$ does not satisfy a nontrivial identity within the given variety;
\item any free group in the variety embeds in an ultrapower of $G$;
\item any free group in the variety embeds in a group elementarily equivalent to $G$.
\end{enumerate}
\end{corollary}

The restriction on $p$ is stipulated, of course, by the celebrated
Novikov--Adian solution of the Burnside problem (see \cite{adyan}):
the Burnside group $B(n,m)$ of exponent $m$ freely generated by $n$ elements,
is infinite for odd $m\ge 665$, and $673$ is the next prime after $665$.

\begin{proof}
Follows from Corollary \ref{no-ident-universal}, Lemma \ref{p}, 
and the fact that the corresponding free groups satisfy the hypothesis of 
Lemma \ref{p}:
from \cite[Chapter VI, \S 3, Theorem 3.3]{adyan} it follows that under the given 
restriction on $p$, any abelian subgroup of the Burnside group $B(n,p)$ is cyclic 
of order $p$, 
and from the Nielsen--Schreier theorem about freeness of subgroups of free 
groups it follows that any abelian subgroup of an absolutely free group
is infinite cyclic.
\end{proof}

The same remarks as those after Corollary \ref{no-ident} apply also in the group case. 
In particular, condition (ii) is equivalent to the condition that the universal theory 
of $G$ is contained in the universal theory of any free group.
As in the case of Lie and associative algebras, ``any free group'' in conditions
(ii) and (iii) of the corollary, can be replaced by ``the free group of rank $2$'':
an embedding of the free group of countable rank into the free group of rank $2$
is well-known in the case of absolutely free groups, and is established in 
\cite{shirv} (see also \cite[\S 3]{atabekyan} and references therein for the 
later alternative approaches) in the case of Burnside groups.

Corollary \ref{cor2} is probably known to experts -- 
at least statements equivalent to the implication (i) $\Rightarrow$ (ii) in the case
of the variety of all groups can be found in 
\cite[Theorem 1]{boffa} and \cite[Lemma 6.15]{drutu-s}. The proofs there are 
different and based on the fact that ultraproducts are $\omega_1$-compact.

\begin{question}
Is there a semigroup property such that the corresponding analogs of the results
of this section would hold in the class of all semigroups? 
The same question for the classes of inverse semigroups and quasigroups.
\end{question}

Note that neither of these classes is ideal-determined.

\section{Application: PI algebras}\label{pi}

As an application of this machinery, let us demonstrate how one can handle, in a 
way different from the traditional approaches, some well-known statements from the 
theory of associative algebras satisfying polynomial identities (usually called
PI algebras).

The Regev celebrated ``$A\otimes B$'' theorem asserts that the tensor product
of two PI algebras $A$ and $B$ is PI (see, for example, 
\cite[Theorem 5.42]{belov-rowen}).
If we want to prove it using results of \S \ref{sect-univalg}, we encounter a few
difficulties: first, to establish relationship between
the ultrapower of the tensor product $(A \otimes B)^{\mathscr U}$ and the tensor product of
ultrapowers $A^{\mathscr U} \otimes B^{\mathscr U}$ (perhaps, considering some sort of 
completed tensor product instead of the usual one may help), and, second, to be able to say
something about algebras $A$ and $B$ such that their (possibly completed) tensor product 
contains a free associative algebra. But at least, in this way we are able to provide
an alternative proof of the particular case where one of the
tensor factors is finite-dimensional (first established by Procesi and Small in 
\cite{ps} for the even more particular case where one of the tensor factors is a
full matrix algebra; a quantitative refinement (degrees of the corresponding
standard identities) of Procesi--Small result was later
given in \cite[Theorem 3.2]{domokos}).
This particular case is morally important, as semiprime PI algebras embed
in matrix algebras over commutative rings (see, for example, 
\cite[Remark 1.69]{belov-rowen}),
which essentially reduces the semiprime situation to a finite-dimensional one.

\begin{theorem}[``Baby Regev's $A\otimes B$'']\label{regev}
The tensor product of two associative algebras, one of them PI and the other 
finite-dimensional, is PI.
\end{theorem}

\begin{lemma}\label{otimes}
Let $A$, $B$ be algebras defined over a field $K$, $A$ finite-dimensional. 
Then, for any ultrafilter
$\mathscr U$, 
$$
(A \otimes_K B)^{\mathscr U} \simeq A \otimes_K B^{\mathscr U}
$$
(as $K$-algebras).
\end{lemma}

Special cases of this assertion were proved many times in literature --
for example, in \cite{nita} for the case where $A$ is an associative full 
matrix algebra, and in \cite[Proposition 25]{taha} for the case where both $A$ 
and $B$ are finite-dimensional, and the proof is standard.

\begin{proof}
Let $\{a_1, \dots, a_n\}$ be a basis of $A$. Obviously, for each $K$-algebra $C$,
each element of $A \otimes_K C$ can be uniquely represented as 
$\sum_{k=1}^n a_k \otimes c_k$ for some $c_k \in C$. Define a map 
$$
\varphi: (A \otimes_K B)^{\mathfrak I} \to A \otimes_K B^{\mathfrak I}
$$
as follows: for $f\in (A\otimes_K B)^{\mathfrak I}$
write $f(i) = \sum_{k=1}^n a_k \otimes b_{ki}$, $i\in \mathfrak I$ and define
$\varphi(f) \in A \otimes_K B^{\mathfrak I}$ as $\sum_{k=1}^n a_k \otimes g_k$, where
$g_k \in B^{\mathfrak I}$ is defined as $g_k(i) = b_{ki}$, $i\in \mathfrak I$.
Writing multiplication in $A$ in terms of the basis elements, one can see that
$\varphi$ is an isomorphism of $K$-algebras.

The ideal $\mathcal I((A\otimes_K B)^{\mathfrak I}, \mathscr U)$ maps under $\varphi$ to 
$A \otimes_K \mathcal I(B^{\mathfrak I}, \mathscr U)$, so factoring out both sides of the 
isomorphism $\varphi$ by the corresponding ideals, we get:
\begin{multline*}
(A\otimes_K B)^{\mathscr U} = 
(A\otimes_K B)^{\mathfrak I} / \mathcal I((A\otimes_K B)^{\mathfrak I}, \mathscr U) \simeq
(A \otimes_K B^{\mathfrak I}) / (A \otimes_K \mathcal I(B^{\mathfrak I}, \mathscr U)) \\ \simeq
A \otimes_K (B^{\mathfrak I} / \mathcal I(B^{\mathfrak I}, \mathscr U)) =
A \otimes_K B^{\mathscr U}.
\end{multline*}
\end{proof}

\begin{corollary}[\protect{\cite[Proposition 21]{taha}}]\label{fin-dim}
Let $A$ be a finite-dimensional algebra defined over a field $K$. Then, for any 
ultrafilter $\mathscr U$, $A^{\mathscr U} \simeq A \otimes_K K^{\mathscr U}$ (as 
$K$-algebras).
\end{corollary}

\begin{proof}
Put $B = K$.
\end{proof}

\begin{proof}[Proof of Theorem \ref{regev}]
Let $A$ be a finite-dimensional associative algebra, and $B$ a PI algebra, 
defined over a field $K$. Suppose $A \otimes_K B$ is not PI. 
Then by Corollary \ref{no-ident}, some ultrapower $(A \otimes_K B)^{\mathscr U}$,
considered as a $K^{\mathscr U}$-algebra, contains a free associative
subalgebra $\mathcal F$ of finite rank. 
By Lemma \ref{otimes}, 
$$
(A \otimes_K B)^{\mathscr U} \simeq 
(A \otimes_K K^{\mathscr U}) \otimes_{K^{\mathscr U}} B^{\mathscr U}
$$
as $K^{\mathscr U}$-algebras. 

Since $\mathcal F$ is finitely-generated, we may choose a finitely-generated 
$K^\mathscr U$-subalgebra $B^\prime$ of $B^{\mathscr U}$ such that $\mathcal F$ is a subalgebra
of $(A \otimes_K K^{\mathscr U}) \otimes_{K^{\mathscr U}} B^\prime$.
Since $B$ is PI, $B^{\mathscr U}$ is PI, and $B^\prime$ is PI.
The Shirshov height theorem implies that $B^\prime$ has polynomial growth (or, in
other words, its Gelfand-Kirillov dimension is finite; see, for example, proof of 
Theorem 9.19 in \cite{belov-rowen}). As $A \otimes_K K^{\mathscr U}$ is 
finite-dimensional (over $K^{\mathscr U}$), the tensor product 
$(A \otimes_K K^{\mathscr U}) \otimes_{K^{\mathscr U}} B^\prime$ has polynomial growth
too. But this contradicts the fact that its subalgebra $\mathcal F$ has
exponential growth.
\end{proof}

Needless to say, this proof, unlike those in \cite{ps}, as well as all the proofs of the 
full-fledged Regev's $A\otimes B$, is absolutely non-constructive, as it uses existence 
of an ultrafilter, and, therefore, axiom of choice.

Along the same lines one may treat, at least in some particular cases,
a number of other well-known results from PI theory:
commutativity of an ordered PI algebra;
PIness of a finitely-graded algebra with PI 
``null component''; of an algebra with a group action whose fixed point subalgebra is PI;
of a localization of a PI algebra; of algebras with involution, etc.

Let us note yet another immediate application to a situation which resonates with the 
universal-algebraic setup of \S \ref{sect-univalg}. A number of authors considered
a property of \textit{definable principal congruences} in algebraic systems. For rings,
this amounts to saying that the property that an element $x$ of a ring belongs to the
(principal) ideal generated by an element $y$, is expressed as a first-order formula in 
free variables $x,y$ of the theory of rings. 
For an associative ring $R$, the fact that $Var(R)$ has
definable principal congruences, can be expressed as a certain formula of 
the universal theory of rings, evidently not satisfied in free associative rings 
(see \cite[Lemma 1]{simons}).  
From this and from the ring-theoretic version of Corollary \ref{no-ident} immediately 
follows that if $Var(R)$ has definable principal
congruences, then $R$ is PI, what is the main result of \cite{simons}, obtained there
with appeal to some deep result from PI theory.

One may try to apply the same reasoning to the known open problem: 
suppose an associative algebra $R$ is represented as the vector space sum of its subalgebras:
$R = A + B$. If $A$, $B$ are PI, is it true that $R$ is PI? (see \cite{1} and \cite{2} with a transitive closure of references 
therein). It is easy to see
that the operation of taking ultraproduct commutes with the operation of taking the
vector space sum: 
$R^{\mathscr U} = A^{\mathscr U} + B^{\mathscr U}$. Consequently, the question
can be reduced to the following one: is it possible that the vector space sum $A + B$
of two PI algebras can contain a free associative subalgebra? If $A + B$ is 
finitely-generated, the impossibility of this follows from the same growth argument 
as in the proof of Theorem \ref{regev}. The difficulty, however, lies in the fact that 
it is not clear how to reduce the situation to a finitely-generated one, as 
the multiplication between $A$ and $B$ can be intertwined in a complicated way.

\section{Application: algebras with the same identities}\label{same-ident}

In \cite[\S 5]{razmyslov} (see also \cite{shest-zaicev}), Kushkulei and 
Razmyslov obtained results stating that some classes of 
fi\-ni\-te-di\-men\-si\-o\-nal algebras (e.g., prime 
over an algebraically closed field) are uniquely determined by their identities.
Another result in this direction:

\begin{theorem}
Let $\mathscr P$ be a class of finite-dimensional algebras satisfying the following
conditions:
\begin{enumerate}
\item 
If $A, B\in \mathscr P$, $A$ and $B$ are defined over the same field, $A$ is a 
subalgebra of $B$, and $Var(A) = Var(B)$, then $A = B$.
\item 
$\mathscr P$ is closed under elementary equivalence in the first-order two-sorted theory of 
pairs (algebra over a field, field).
\item $\mathscr P$ contains all finite-dimensional prime algebras.
\end{enumerate}
Then $\mathscr P$ satisfies the following strengthening of condition (i):
if $A, B\in \mathscr P$, $A$ and $B$ are defined over the same field, and 
$Var(A) = Var(B)$, then $A \simeq B$.
\end{theorem}

We stress that the base field over which algebras in the class $\mathscr P$ are 
defined, is not fixed.

\begin{proof}
Let $A, B\in \mathscr P$, both defined over a field $K$, be such that 
$Var(A) = Var(B)$. 
Lemma \ref{zariski}, Corollary \ref{cor} and Corollary \ref{fin-dim} imply that a
free algebra $\mathcal F$ in the variety $Var(A) = Var(B)$ embeds,
as a $K$-algebra, into an algebra $A \otimes K^{\mathscr U}$ for some 
ultrafilter $\mathscr U$. Hence $K^{\mathscr U} \mathcal F$ is isomorphic, as a 
$K^{\mathscr U}$-algebra, to a subalgebra of $A \otimes K^{\mathscr U}$. It is easy to see
that primeness of the $K$-algebra $\mathcal F$ implies primeness of the 
$K^{\mathscr U}$-algebra $K^{\mathscr U} \mathcal F$. 
By (iii), the latter algebra belong to $\mathscr P$, and
by the {\L}o\'s theorem and (ii), the $K^{\mathscr U}$-algebra 
$A \otimes_K K^{\mathscr U}$ belongs to $\mathscr P$.
Obviously, $K^{\mathscr U} \mathcal F$ satisfies over $K^{\mathscr U}$
the same identities as $A$ over $K$. Then by (i), 
$K^{\mathscr U} \mathcal F \simeq A \otimes_K K^{\mathscr U}$.
By the same reasoning, 
$K^{\mathscr U} \mathcal F \simeq B \otimes_K K^{\mathscr U}$.
Hence $A \otimes_K K^{\mathscr U} \simeq B \otimes_K K^{\mathscr U}$ as 
$K^{\mathscr U}$-algebras, and $A \simeq B$ as $K$-algebras.
\end{proof}

This theorem allows, for example, to obtain an alternative proof of results
of Kushkulei and Razmyslov in an important particular case of finite-dimensional simple Lie algebras
over an algebraically closed field of characteristic zero 
(see \cite[\S5, Corollary 1 and Comments]{razmyslov}). Indeed, the hypothesis
of the theorem is satisfied for such class of algebras: (i) can be proved
with the help of the well-known Dynkin's classification \cite{dynkin} of semisimple
subalgebras of semisimple Lie algebras, (ii) is evident, and (iii) follows from
the obvious fact that for finite-dimensional Lie algebras over a field of 
characteristic zero, simplicity is equivalent to primeness.

The same approach can be applied to finite-dimensional simple Jordan algebras
(which follows from the general results of Kushkulei and Razmyslov and also established independently
in \cite{drensky-racine}), as well as to graded simple associative algebras
(\cite{koshlukov-zaicev}).

Along the same lines one may treat varieties generated by affine Kac--Moody algebras.
Consider, for example, a Lie algebra of the form 
\begin{equation}\label{km}
\widehat{\mathfrak g} = (\mathfrak g \otimes_K K[t,t^{-1}]) \oplus Kz ,
\end{equation}
where $\mathfrak g$ is a split finite-dimensional
simple Lie algebra defined over a field $K$ of characteristic zero, 
$K[t,t^{-1}]$ is the algebra of Laurent polynomials,
$z$ is the central element, and the multiplication between elements of
$\mathfrak g \otimes K[t,t^{-1}]$ is twisted by the well-known $2$-cocycle:
$$
[x \otimes f, y \otimes g] = [x,y] \otimes fg + (x,y) Res (\frac{df}{dt} g) z
$$
where $x,y\in \mathfrak g$, $f,g\in K[t,t^{-1}]$, and $\form$ is the Killing form on 
$\mathfrak g$.

In \cite{zaicev} it is shown, among other, that 
$Var(\widehat{\mathfrak g}) = \Lbrack Var(\mathfrak g),\mathfrak E \Rbrack$, where 
$\mathfrak E$ is the variety consisting of the single zero algebra, and 
$\Lbrack\,\cdot\, , \cdot\,\Rbrack$ is the standard commutator
of varieties as defined in \cite[\S 4.3.8]{bahturin} (in other words, for a variety
$\mathfrak V$, $\Lbrack\mathfrak V, \mathfrak E\Rbrack$ is nothing but a variety defined by identities of the form 
$[f(x_1, \dots, x_n),x_{n+1}] = 0$, where $f(x_1, \dots, x_n) = 0$ is an identity in
$\mathfrak V$). Let us complement this result by showing that free algebras
in $Var(\widehat{\mathfrak g})$ embed in algebras whose structure closely 
resembles those of $\widehat{\mathfrak g}$.

By \cite[Corollary 2.2]{stewart}, every ideal of the Lie algebra 
$\mathfrak g \otimes K[t,t^{-1}]$ is of the form $\mathfrak g\otimes I$, where $I$
is an ideal of $K[t,t^{-1}]$. Since $K[t,t^{-1}]$ does not have zero divisors, it is
prime, hence $\mathfrak g \otimes K[t,t^{-1}]$ is prime, and $\widehat{\mathfrak g}$ is 
a central extension of a prime Lie algebra. 
By an obvious modification of the proof of Lemma \ref{zariski}, one get that 
for the free algebra $\mathcal L$ in $Var(\widehat{\mathfrak g})$ 
of countable rank, $\mathcal L/Z(\mathcal L)$ is prime, and, by Corollary \ref{cor} and Lemma \ref{otimes}, embeds in
$\mathfrak g \otimes K[t,t^{-1}]^{\mathscr U}$ for some ultrafilter $\mathscr U$.
From the results of \cite[\S 4.4]{bahturin} it follows that $\mathcal L$ embeds
in a central extension of $\mathfrak g \otimes K[t,t^{-1}]^{\mathscr U}$.
The latter, by \cite[Theorem 3.3 and Corollary 3.5]{kassel}, is described in terms
of the first-order cyclic homology of $K[t,t^{-1}]^{\mathscr U}$, so we get an embedding
$$
\mathcal L \>\hookrightarrow\> \Big(\mathfrak g \otimes_{K^{\mathscr U}} K[t,t^{-1}]^{\mathscr U}\Big) 
\oplus HC_1(K[t,t^{-1}]^{\mathscr U}) .
$$
The multiplication in the right-hand side Lie algebra is defined by the formula
$$
[x \otimes F, y \otimes G] = [x,y] \otimes FG + (x,y) \overline{F \wedge G} ,
$$
where $x,y\in \mathfrak g$, $F,G \in K[t,t^{-1}]^{\mathscr U}$, and 
$\overline{F \wedge G}$ denotes the corresponding homology class in \newline
$HC_1(K[t,t^{-1}]^{\mathscr U})$.

The addition to (\ref{km}) of the $Kt\frac{d}{dt}$ term, or twisting by automorphisms
of $\mathfrak g$, do not significantly change the picture, and can be treated in
the same way.

Similarly, one may treat varieties generated by modular semisimple Lie algebras. According
to the classical Block theorem, a typical finite-dimensional semisimple Lie
algebra over a field $K$ of characteristic $p$ which is not isomorphic to the sum
of simple ones, has the form 
$$
\Big(S \otimes_K K[t_1, \dots, t_n]/(t_1^p, \dots, t_n^p)\Big) \oplus 
\Big(1\otimes_K D\Big) ,
$$
where $S$ is a simple Lie algebra, $K[t_1, \dots, t_n]/(t_1^p, \dots, t_n^p)$ is the
reduced polynomial algebra, and $D$ is a derivation algebra of 
$K[t_1, \dots, t_n]/(t_1^p, \dots, t_n^p)$ such that the latter does not have
$D$-invariant ideals.

To such algebras, Lemma \ref{zariski} is applicable, and, as in the Kac--Moody case,
further application of Corollary \ref{cor} and Lemma \ref{otimes} gives an embedding
(as $K^{\mathscr U}$-algebras)
$$
\mathcal L \>\hookrightarrow\> 
\Big(S \otimes_{K^{\mathscr U}} K^{\mathscr U}[t_1, \dots, t_n]/(t_1^p, \dots, t_n^p)\Big) 
\oplus \Big(1 \otimes_{K^{\mathscr U}} (D \otimes_K K^{\mathscr U})\Big) 
$$
for some ultrafilter $\mathscr U$.

\section{Application: Tarski's monsters}\label{tarski}

Under \textit{Tarski's monster of type $p$}, $p$ being a prime (respectively, 
\textit{of type $\infty$}) we understand an infinite nonabelian group
all whose proper subgroups are cyclic of order $p$ (respectively, of infinite 
order). Such groups were constructed, among other groups
with exotic-looking restrictions on subgroups, by Olshanskii in the framework of his 
celebrated machinery of geometrically-motivated manipulations with group presentations
(see, for example, \cite[Chapter 9, \S 28.1]{olsh-book}). In Olshanskii's works,
the existence of Tarski's monsters of type $p$ is established for $p > 10^{75}$. 
Later, in \cite{al} this estimate has been reduced to $p > 1003$.

Let $G$ be a finitely-generated group, and
\begin{equation}\label{present}
\{1\} \to \mathcal N \to \mathcal F \to G \to \{1\}
\end{equation}
its presentation, where $\mathcal F$ is a free group of finite rank, and 
$\mathcal N$ is a normal subgroup of relations. The \textit{girth} of the presentation
(\ref{present}) is the minimal length of elements of $\mathcal N$, i.e., the minimal
length of relations between the chosen generators of $G$ (or, in other words,
the minimal length of a simple loop in the corresponding Cayley graph). 
The \textit{girth} of $G$ is the supremum of girths of all its presentations with a finite
number of generators. This natural notion was 
introduced and studied recently by Akhmedov in \cite{akhmedov-1} and \cite{akhmedov-2},
and by Schleimer in \cite{schleimer}.
One of the interesting questions arising in that regard is to construct groups 
of infinite girth.

As noted in the above-mentioned works, a group satisfying a nontrivial identity 
cannot have an infinite girth. To circumvent this obstacle, let us introduce the 
notion of \textit{relative girth} -- a girth relative to all identities a group 
satisfies: in the definition of girth above, replace in (\ref{present}) the 
absolutely free group $\mathcal F$ by a free group (of finite rank) in the variety 
$Var(G)$.

\begin{theorem}\label{monster}\hfill
\begin{enumerate}
\item 
A Tarski's monster of type $p$ does not satisfy any nontrivial identity
except $x^p = 1$ and its consequences, if and only if it has infinite relative girth.
\item
A Tarski's monster of type $\infty$ does not satisfy any nontrivial identity
if and only if it has infinite girth.
\end{enumerate}
\end{theorem}

\begin{remark}[A. Olshanskii]
Tarski's monsters satisfying condition (ii) of the theorem do exist 
(and, moreover, there is an abundance of them), and they can be constructed in the 
following way.

According to \cite[Corollary 1]{residualing}, each non-cyclic torsion-free hyperbolic
group $G_0$ has a homomorphic image $G$ which is a Tarski's monster of type $\infty$.
Such monsters are constructed by subsequent applications of \cite[Theorem 2]{residualing},
as the direct limit of a system of surjective maps of groups $G_0 \to G_1 \to \dots$.
Each $G_n$ is a non-cyclic torsion-free hyperbolic group (and hence contains a 
free subgroup of countable rank), and is $2$-generated for 
$n\ge 1$. Let us denote, by abuse of notation,
the corresponding generators by the same letters $a,b$ 
(so, $a,b\in G_n$ are images of $a,b\in G_{n-1}$), and each $G_n$ is obtained from 
$G_{n-1}$ by adding additional relations between these two generators.
Also, the injectivity radius of each surjection $G_{n-1} \to G_n$ (i.e., the maximal 
number $r$ such that the map is injective on all words of length $\le r$) can be chosen 
to be arbitrarily large. 

Enumerate all the non-trivial words in the free group of countable rank $\mathcal F$ as 
$v_1, v_2, \dots$. Since each $G_n$ contains a copy of $\mathcal F$, there are
elements in $G_n$ such that the value of $v_n$ on these elements is different from $1$.
Writing these elements in terms of the generators $a,b$, we get
\begin{equation}\label{w}
w_n(a,b) = v_n (w_{n1}(a,b), w_{n2}(a,b), \dots) \ne 1 \text{ in } G_n
\end{equation}
for some (finite number of) words $w_n, w_{n1}, w_{n2}, \dots$.

Now, on each step choose the injectivity radius of the surjection $G_n \to G_{n+1}$ 
larger than the length of all words $w_1, \dots, w_n$ constructed on the previous steps.
Consequently, (\ref{w}) holds in all groups $G_{n+1}, G_{n+2}, \dots$, 
and hence in the limit group $G$. This implies that $v_n = 1$, for any $n$, cannot be an 
identity of $G$.
\end{remark}

In \cite{garion}, the absence of nontrivial identities in a Tarski's monster $G$
of type $\infty$ is characterized in terms of an action of the group of outer 
automorphisms of a free group of rank $n$ on $n$-tuples of $G$.

\begin{question}
Prove existence of Tarski's monsters satisfying condition (i) of Theorem \ref{monster}.
\end{question}

\begin{proof}[Proof of Theorem \ref{monster}]
The ``only if'' part is obvious, so let us prove the ``if'' part.
Let $G$ be Tarski's monster either of type $p$ which does not satisfy any nontrivial 
identity except $x^p = 1$ and its consequences, or of type $\infty$ 
which does not satisfy a nontrivial identity.
By Corollary \ref{cor2}, a group elementarily equivalent to $G$ contains a subgroup 
isomorphic to a relatively free subgroup $\mathcal G$ of rank $2$ (which is the free 
Burnside group $B(2,p)$ in the case (i), or the free group in the case (ii)). 
Let $x,y$ be the free generators of $\mathcal G$, and
$$
\{w_1(x,y) = x, w_2(x,y) = x^{-1}, w_3(x,y) = y, w_4(x,y) = y^{-1}, \dots, w_{k_n}(x,y)\}
$$
the set of all words of $\mathcal G$ of length $\le n$. 
The existence of this ``initial piece of $\mathcal G$ of 
length $n$'' can be written as the first-order property:
$$
\exists x \> \exists y: \bigwedge_{1 \le i < j \le k_n} w_i(x,y) \ne w_j(x,y) .
$$
Consequently, for each $n\in \mathbb N$, the same first-order formula holds in $G$;
let $x_n, y_n\in G$ be the corresponding elements. Obviously,
$x_n, y_n$ do not commute except, possibly, for some small values of $n$, and, therefore,
generate $G$. This provides a presentation of $G$ of (relative) girth $>n$.
\end{proof}

Note another interesting consequence of Theorem \ref{monster}.

The \textit{growth sequence} of a group $G$ is a sequence whose $n$th term equal to the 
(minimal) number of generators of the $n$th fold direct power of $G$.
See \cite{wise} for a brief history of the subject and further references.
In particular, in a number of works, including \cite{wise}, a considerable effort was put
into construction of groups whose growth sequence is constant, each term is equal to $2$.
Theorem \ref{monster} provides further such examples, in view of the following general 
elementary fact:

\begin{lemma}\label{univ}
If a finitely-generated simple group $G$ has infinite relative girth, then its growth 
sequence is constant, each term is equal to the minimal number of generators of $G$.
\end{lemma}

\begin{proof}
Let $n$ be the minimal number of generators of $G$.
Infinity of the relative girth of $G$ means that there is an infinite sequence
$\mathcal N_1, \mathcal N_2, \dots$ of normal subgroups of the free group 
$\mathcal G$ in $Var(G)$ of rank $n$ such that for every $i$
the length of each word in $\mathcal N_i$ is $\ge i$, and 
$\mathcal G/\mathcal N_i \simeq G$.

Let us prove by induction that for each $k\in \mathbb N$ there is a sequence
$i_1 < i_2 < \dots < i_k$ such that
\begin{equation}\label{step}
\mathcal G/(\mathcal N_{i_1} \cap \dots \cap \mathcal N_{i_k}) 
\simeq 
G \times G \times \dots \times G \quad (k \text{ times}) .
\end{equation}
For $k=1$ we may take $i_1 = 1$. Suppose that for some $k > 1$ the isomorphism 
(\ref{step}) holds. It is obvious that 
$\mathcal N_{i_1} \cap \dots \cap \mathcal N_{i_k} \ne \{1\}$.
On the other hand, since $\bigcap_{i > i_k} \mathcal N_i = \{1\}$, there is 
$i_{k+1} > i_k$ such that 
$$
\mathcal N_{i_1} \cap \dots \cap \mathcal N_{i_k} \not\subseteq
\mathcal N_{i_{k+1}} .
$$
Since $G$ is simple, $\mathcal N_{i_{k+1}}$ is a maximal normal subgroup in $\mathcal G$,
and 
$$
\mathcal G = 
(\mathcal N_{i_1} \cap \dots \cap \mathcal N_{i_k}) 
\mathcal N_{i_{k+1}} .
$$
Then:
\begin{multline*}
\mathcal G/(\mathcal N_{i_1} \cap \dots \cap \mathcal N_{i_k} \cap \mathcal N_{i_{k+1}}) 
\simeq 
\mathcal G/(\mathcal N_{i_1} \cap \dots \cap \mathcal N_{i_k}) \times 
\mathcal G/\mathcal N_{k+1} \\ \simeq 
G \times \dots \times G \quad (k+1 \text{ times}) .
\end{multline*}
It follows from (\ref{step}) that each finite direct power of $G$ is $n$-generated.
\end{proof}

In fact, nothing in this proof is specific to groups: the corresponding statement can
be formulated for general algebraic systems; in particular, it holds also for algebras.

Lemma \ref{univ} implies that the growth sequence of Tarski's monsters 
satisfying conditions of Theorem \ref{monster} is constant, each term is equal 
to $2$. Note that in \cite[Theorem 8]{garion} it is proved, by different methods,
that each term in the growth sequence of \textit{any} Tarski's monster 
is $\le 3$, and for any Tarski's monster of type $\infty$ it is equal to $2$.

\section{Homomorphisms of direct products}\label{dual}

The following is, in a sense, dual version of Theorem \ref{ultra},
with embeddings replaced by surjective homomorphisms, with, essentially, the
same proof\footnote[2]{
Added August 22, 2015: the proof of Theorem \ref{th-dual} is in error. For a 
corrected statement, with a different proof, see a sequel to this article,
\emph{On the utility of Robinson--Amitsur ultrafilters. II}.
}.

\begin{theorem}\label{th-dual}
Let $\{ B_i \}_{i\in \mathfrak I}$ be a set of algebraic systems from an 
ideal-determined class. 
If a finitely subdirectly irreducible algebraic system $A$ is a surjective 
homomorphic image of the direct product $\prod_{i\in \mathfrak I} B_i$, then there 
is an ultrafilter $\mathscr U$ on the set $\mathfrak I$ such that $A$ is a
surjective homomorphic image of the ultraproduct $\prod_{\mathscr U} B_i$.
\end{theorem}

\begin{proof}
Let $\alpha: \prod_{i\in \mathfrak I} B_i \to A$ be a given surjective 
homomorphism. Define 
$$
\mathscr S = 
\set{\set{i\in \mathfrak I}{f(i)\ne e}}
{f\in \prod_{i\in \mathfrak I} B_i, \>\alpha(f) \ne e} .
$$
Let us verify that intersection of any two elements of $\mathscr S$ 
contains an element of $\mathscr S$. 
Let $S,T\in \mathscr S$, say, $S = \set{i\in \mathfrak I}{f(i)\ne e}$
and $T = \set{i\in \mathfrak I}{g(i)\ne e}$ for some 
$f, g\in \prod_{i\in \mathfrak I} B_i$ not vanishing under $\alpha$.
Since $A$ is finitely subdirectly irreducible, it contains an element $a\ne e$
belonging to the intersection of ideals generated by $\{\alpha(f)\}$ and 
$\{\alpha(g)\}$. Take some element $u\in \prod_{i\in \mathfrak I} B_i$ from 
the preimage $\alpha^{-1}(a)$.
Let $i\in \mathfrak I$ such that $f(i) = e$.
Since $\alpha(u) = a = t(b_1, \dots, b_n, \alpha(f), \dots, \alpha(f))$ for some
ideal term $t$ and $b_1, \dots, b_n\in A$, we have
$u = t(h_1, \dots, h_n, f, \dots, f)$ for some 
$h_1\in \alpha^{-1}(b_1), \dots, h_n\in \alpha^{-1}(b_n)$, and
\begin{multline*}
u(i) = t(h_1, \dots, h_n, f, \dots, f)(i) = 
t(h_1(i), \dots, h_n(i), f(i), \dots, f(i)) \\ = t(h_1(i), \dots, h_n(i), e, \dots, e)
= e .
\end{multline*}
Coupling this with a similar assertion for $g$, we get that
$$
S \cap T \supset \set{i\in \mathfrak I}{u(i) \ne e} \in \mathscr S.
$$
 
Thus $\mathscr S$ satisfies the finite intersection property and is contained in
some ultrafilter $\mathscr U$ on $\mathfrak I$. 
If $f \in \mathcal I\Big(\prod_{i\in \mathfrak I} B_i, \mathscr U\Big)$,
then $\set{i\in \mathfrak I}{f(i) = e} \in \mathscr U$, and, since
$\mathscr U$ is ultrafilter, 
$\set{i\in \mathfrak I}{f(i) \ne e} \notin \mathscr U$, 
and hence $\set{i\in \mathfrak I}{f(i) \ne e} \notin \mathscr S$ and 
$\alpha(f) = e$.
This shows that $\mathcal I\Big(\prod_{i\in \mathfrak I} B_i, \mathscr U\Big)$ 
lies in $\Ker\alpha$, and the map $\alpha$ factors through the ultraproduct
$
\Big(\prod_{i\in \mathfrak I} B_i\Big) \Big/ \mathcal I\Big(\prod_{i\in \mathfrak I} B_i, \mathscr U\Big)
= \prod_{\mathscr U} B_i$.
\end{proof}

In \cite[\S 5.3]{zilber}, it is asked whether an ultraproduct of finite groups
can be mapped surjectively to a compact simple Lie group $G$, in particular,
$\SO(3)$. According to Theorem \ref{th-dual},
this is equivalent to the question whether a direct product of
finite groups can be mapped surjectively to $G$.

\begin{question}
What can be said about the class of groups obtained by closure of finite groups
with respect to direct products and homomorphic images? Which simple groups, 
matrix groups, Lie groups, lie in this class?
\end{question}

A question when a homomorphic image of a direct product of groups or algebras
can be factored through an ultraproduct, or, more generally, through a product 
of a finite number of ultraproducts, was studied also in the recent interesting 
papers \cite{bergman} and \cite{bergman-nahlus}. The emphasis there is, however,
different: one imposes various conditions 
(such as simplicity, solvability, or nilpotency) not on $A$ but on $B_i$'s.

\section{Further speculations}

Here we indicate some of our initial reasons for looking into all this, of a 
highly speculative character.

\subsection{Lie-algebraic monsters}

A problem of existence of Lie-algebraic analogs of Tarski's monsters -- namely, 
of in\-fin\-ite-di\-men\-si\-o\-nal Lie algebras all whose proper subalgebras 
are one-di\-men\-si\-o\-nal -- is, arguably, one of the most difficult problems in the 
abstract theory of infinite-dimensional Lie algebras (see, for example,
\cite{gein} and \cite{lashi}). For such hypothetical
Lie algebras which do not satisfy a nontrivial identity, an analog of 
Theorem \ref{monster} would hold.

\begin{question}
Study the notion of (relative) girth for Lie algebras.
\end{question}

A simpler, but still open problem concerns \textit{finite-dimensional} 
Lie algebras all whose proper subalgebras are one-dimensional (called
\textit{$L_1$-algebras} in the sequel). 
Over perfect fields, every $L_1$-algebra has dimension $2$ or $3$.
In characteristic zero this is an immediate consequence of simple reasonings 
about forms of algebras and the classical structure theory, and in the case of 
the perfect base field of characteristic $p>0$ this follows from the powerful 
results of Premet \cite{premet} (for $p>5$) and \cite{premet-p5} (for $p=5$). 
See, for example, \cite[Proposition 1]{gein} or \cite[Proposition 3.1]{lashi} 
for details. 

Over non-perfect fields, we do not have the usual tools of Galois theory
to handle forms of algebras, as in \cite{premet}, and, to quote \cite{varea}, 
``the existence [of an $L_1$-algebra] of dimension greater than $3$
is an interesting open problem''.
Using reasonings similar to those in \S \ref{sect-alg}, it is possible to prove 
that a prime finite-dimensional Lie algebra over an infinite field, all whose
abelian subalgebras are $1$-dimensional, satisfies a functional identity of the 
form
\begin{equation}\label{quasi}
\lambda(x,y) u(x,y) + \mu(x,y) v(x,y) = 0 ,
\end{equation}
where $u$ and $v$ are certain Lie words in $2$ variables, and 
$\lambda$ and $\mu$ are elements of the base field, depending, in general, 
on the elements $x$ and $y$ of the algebra.
We conjecture that no simple finite-dimensional Lie algebra of dimension $>3$
can satisfy a nontrivial functional identity of the form (\ref{quasi}). 
If this conjecture is true, it will prove the absence of $L_1$-algebras of 
dimension $>3$ over an arbitrary base field.

\subsection{The Tits alternative}

Another reason was the desire to provide an alternative proof of the celebrated Tits 
alternative, or for one of its not less celebrated consequences, such as growth dichotomy for
linear group.
The Tits alternative claims that a linear group contains either a solvable subgroup of finite 
index, or a nonabelian free subgroup. 
A nice and important, yet admitting a few-lines elementary proof, result of Platonov \cite{platonov}
states that a linear group which has a nontrivial identity, contains a solvable subgroup of
finite index. Modulo this result, the proof of the Tits alternative reduces to establishing
that a linear group $G$ which does not satisfy a nontrivial identity, contains a nonabelian free
subgroup. One may naively argue as follows:
by Corollary \ref{cor2}, a group, elementarily equivalent to $G$,
contains a nonabelian free subgroup. From this we may infer some first-order properties of 
$G$, for example, that it contains ``a piece of a nonabelian free group of arbitrarily large
length'', like in the proof of Theorem \ref{monster}.
On the other hand, as linearity is a sort of finiteness condition, 
one may hope that these first-order properties may help to construct a nonabelian free 
subgroup in $G$.
If successful, this approach would provide a proof of the Tits alternative drastically different
from all the proofs given so far.

\subsection{Jacobson's problem}

An old open problem due to Jacobson asks whether a Lie $p$-algebra $L$ such that
for every $x\in L$ there is $n(x)\in \mathbb N$ satisfying
\begin{equation}\label{xp}
x^{p^{n(x)}} = x ,
\end{equation}
is abelian?
As one of the first steps, one may wish to prove that such Lie algebras satisfy
a nontrivial identity; or, for example, that there are no simple (or even prime) such algebras. In both of
these cases, by Lemma \ref{zariski} and Corollary \ref{cor}, a free Lie algebra
embeds in an algebra elementarily equivalent to $L$, over some elementary 
extension of the base field. The condition (\ref{xp}) is not the first-order property
(unless all $n(x)$ are bounded, in which case the problem is trivial), but one
may hope to derive from it some first-order consequences which will come into 
contradiction with the existence of a free Lie subalgebra.

\section*{Acknowledgements}

Thanks are due to Shelly Garion, Misha Gavrilovich, Sergei Ivanov, 
Ruvim Lipyanski, Alexander Olshanskii, Mark Sapir, Ivan Shestakov, 
Irina Sviridova, Jen\H{o} Szigeti, and Boris Zilber for helpful and interesting
comments.

A previous version of this manuscript was rejected by 
\textit{Journal of the London Mathematical Society},
but I owe to the anonymous referee of that submission a great deal of improvements
and corrections, both explicit and implicit: 
in particular, he gave me an idea to use the condition of finite subdirect irreducibility,
and to formulate Theorem \ref{ultra} in the language of universal 
algebra. 
The final touches were done during my visit to IH\'ES at February 2012.

This work was partially supported by grants ERMOS7  
(Estonian Science Foundation and Marie Curie Actions) and ETF9038 
(Estonian Science Foundation).

\end{document}